\documentclass[11pt]{amsart}
\usepackage[hmargin=2.5cm,vmargin=2.5cm]{geometry}
\usepackage{amsfonts}
\usepackage{amsthm}
\usepackage[english]{algorithm2e}
\usepackage{amsmath}
\usepackage{amssymb}
\usepackage{hyperref}
\usepackage[T1]{fontenc}
\newtheorem{Theorem}{Theorem}[section]
\newtheorem{theorem}[Theorem]{Theorem}
\newtheorem{definition}[Theorem]{Definition}
\newtheorem{remark}[Theorem]{Remark}
\newtheorem{example}[Theorem]{Example}
\newtheorem{lemma}[Theorem]{Lemma}
\newtheorem{proposition}[Theorem]{Proposition}

\title{Deformations of relative Rota-Baxter operators on Hom-Jacobi-Jordan algebras}
\author[J. Anith\'eou, S. Attan and K. Kangni]
{Jules Anith\'eou, Sylvain Attan and Kinvi Kangni} 
\address{Jules Anith\'eou \newline
Institut de Math\'ematiques et de Sciences Physiques, Universit\'{e} d'Abomey-Calavi
01 BP 613-Oganla, Porto-Novo, B\'{e}nin}
\email{jules.anitcheou@imsp-uac.org}
\address{Sylvain Attan \newline
 D\'{e}partement de Math\'{e}matiques, Universit\'{e} d'Abomey-Calavi
01 BP 4521, Cotonou 01, B\'{e}nin}
\email{syltane2010@yahoo.fr}

\address{Kinvi Kangni \newline
 Universit\'e F\'elix Houphou\"et-Boigny}
\email{kangnikinvi@yahoo.fr}
\begin{document}
\maketitle
\begin{abstract}
Representations of Hom-Jacobi-Jordan algebras are studied. In particular, adjoint representations and trivial representations are studied in detail.    Derivations and central  extensions of Hom-Jacobi-Jordan algebras are also discussed as an application. Morover, we introduce the cohomology theory on Hom-Jacobi-Jordan algebras as well as the one of relative Rota-Baxter operators
on Hom-Jacobi-Jordan algebras. Finally, we use the cohomological approach to study deformations of Hom-Jacobi-Jordan algebras and those of relative Rota-Baxter operators.
\end{abstract}

\section{Introduction}
 A commutative algebra satisfying the Jacobi identity called a Jacobi-Jordan algebra was introduced first in \cite{kaz}, where an important example of infinite-dimensional solvable but not nilpotent Jacobi-Jordan algebra was given. They are rather special objects in the jungle of non-associative algebras. Different names are used to study these algebras, indeed they are called mock-Lie algebras, Jordan algebras of nil index 3, Lie-Jordan, pathological algebras or Jacobi-Jordan algebras in the literature \cite{absb1, absb2, dbaf, egmk, sonk, sw, awb, kazam}. In contrary to  associative and nonassociative algebras, results on cohomology theories of Jacobi-Jordan algebras are relatively scarce for a long time. Recently, analogous to the existing theories for associative and Lie algebras, cohomology and deformation theories for Jacobi-Jordan algebras are developped \cite{abak} where it is observed that they have several properties not enjoyed by the Hochschild theory. Actually, this cohomology is called a zigzag cohomology since its complex is defined by two sequences of operators.

 Due to the importance of Hom-algebras in several domains, the twisted generalization of Jacobi-Jordan algebra called Hom-Jacobi-Jordan algebra is initiated in \cite{cehgdh} while its representation theory is introduced in \cite{sa}. Roughly, a Hom-type of a given algebra is obtained by a certain
twisting of the defining identities by a twisting map, in such a way that when this twisting map is the identity map, then one recovers the original algebra. The fisrt class of Hom-algebras is the one of Hom-Lie algebras by Hartwig, Larsson and Silverstory \cite{jhdlss} to characterize the structures on deformations of the Witt and the Virasoro algebras. Hom-Jacobi-Jordan algebras are very close to  Hom-Lie algebras having only the skew-symmetry condition replaced by the symmetry condition. Cohomology and deformation of associative and nonassociative algebras have been extended to Hom-associative and non-Hom-associative algebras in many paper these lastest years. In this paper, as Jacobi-Jordan algebras, cohomology theory of Hom-Jacobi-Jordan algebras is introduced and it is observed that it is a zigzag cohomology since the complex is defined by two sequences of operators. 

A description of the rest of this paper is as follows. In the second section  some basic notions and concepts of Hom-Jacobi-Jordan algebras and their representations are given. In Section 3, derivations of  Hom-Jacobi-Jordan algebras are studied and generalized to derivations of  Hom-Jacobi-Jordan algebras with values in representations. For any nonnegative integer $k,$ we define $\alpha^k$-derivations of Hom-Jacobi-Jordan algebras. Several results are proved, in particular, necessary and sufficient condition for the anti-commutator of two $\alpha$-derivations( respectively $\alpha$-antiderivations) to be an $\alpha$-derivations( respectively $\alpha$-antiderivations) are obtained. In Section 4, we introduce a cohomology for Hom-Jacobi-Jordan algebras. As Jacobi-Jordan algebras case, it is called a zigzag cohomology since it deals with two types of cochains and two sequences of operators. As applications, we explore the low degree cohomology spaces and explicit computations on examples are provided. Next, trivial and adjoint representation of Hom-Jacobi-Jordan algebras are studied and some results are obtained. Finally, in Section 5, we study cohomologies of relative Rota-Baxter operators of Hom-Jacobi-Jordan algebra with respect to a representation. Next, linear deformations and formal deformations of Hom-Jacobi-Jordan algebras as well as those of relative Rota-Baxter operators are expored.
\section{Basic results on Hom-Jacobi-Jordan algebras}
Throughout this paper, all vector spaces and algebras are meant over a ground field $\mathbb{K}$.
\begin{definition} A Hom-algebra is a triple $(A,\mu,\alpha)$ in which $(A,\alpha)$ is a Hom-module, $\mu: A^{\otimes 2}\longrightarrow A$ is a linear map.
The Hom-algebra $(A,\mu,\alpha)$ is said to be  multiplicative if $\alpha\circ\mu=\mu\circ\alpha^{\otimes 2}.$  A morphism 
$f: (A,\mu_A,\alpha_A)\longrightarrow (B,\mu_B,\alpha_B)$ of Hom-algebras is a morphism of the underlying Hom-modules such that $f\circ\mu_A=\mu_B\circ f^{\otimes 2}.$
\end{definition}
\begin{definition}
A multiplicative Hom-algebra $(A,\mu,\alpha)$  is said to be Hom-associative if
\begin{eqnarray}
 \mu(\mu(x,y),\alpha(z))=\mu(\alpha(x),\mu(y,z)) \mbox{ $\forall x, y, z\in A.$ } \nonumber
\end{eqnarray}
\end{definition}
\begin{definition}
 A Hom-Jacobi-Jordan algebra is a multiplicative Hom-algebra $(A,\ast,\alpha)$ such that
 \begin{eqnarray}
  && x\ast y=y\ast x \mbox{ ( commutativity)}, \nonumber\\
  && J_{\alpha}(x,y,z):=\circlearrowleft_{(x,y,z)}(x\ast y)\ast\alpha(z)=0, \label{JJi}
 \end{eqnarray}
 where $\circlearrowleft_{(x,y,z)}$ is the sum over cyclic permutation of $x,y,z$ and 
 $J_{\alpha}$ is called the Hom-Jacobian.
\end{definition}
\begin{remark}
 If $\alpha=Id$ (identity map) in a Hom-Jacobi-Jordan algebra $(A,\ast,\alpha),$ then it reduces to a usual Jacobi-Jordan algebra $(A,\ast).$ It follows that the
category of  Hom-Jacobi-Jordan algebras contains the one of Jacobi-Jordan algebras.
\end{remark}
\begin{example}\label{ehjj1}
 Let $A$ be a $3$-dimensional vector space 
 generated by $(e_1, e_2, e_3).$ Then $\mathcal{A}_1:=(A,\ast, \alpha)$ is a Hom-Jacobi-Jordan algebra where $\alpha(e_1)=e_1,\ \alpha(e_2)=2e_2,\ \alpha(e_3)=2e_3$ with the only
non-zero products $e_2\ast e_2=e_3,\ e_2\ast e_3=e_3\ast e_2:=2e_1,\ e_3\ast e_3=2e_3.$ 
\end{example}
We have  the following result.
\begin{proposition}
 Let $(L, \ast, \alpha)$ be a Hom-Jacobi-Jordan algebra and $(A,\bullet, \beta)$ be a commutative Hom-associative algebra. Then  $(L\otimes A, \circ ,\theta:=\alpha\otimes\beta)$ is a Hom-Jacobi-Jordan algebra where
 \begin{eqnarray}
  && (x\otimes a)\circ(y\otimes b)=(x\ast y)\otimes(a\bullet b)  \mbox{  $\forall x,y\in L$ and $a,b\in A, $}\nonumber\\
  && \theta(x\otimes b)=\alpha(x)\otimes \beta(a) \mbox{  $\forall x \in L$ and $a \in A.$}\nonumber 
 \end{eqnarray}
\end{proposition}
\begin{proof}
 First, the multiplicativity of $\theta$ with respect to $\circ$ follows from those of $\alpha$ and $\beta$ with respect to $\ast$ and $\bullet$ respectively. Next, pick $x,y,z\in L$ and $a,b,c\in A.$ Then,
 \begin{eqnarray}
  && \Big((x\otimes a)\circ(y\otimes b) \Big)\circ \theta(z\otimes c)+
  \Big((y\otimes b)\circ(z\otimes c) \Big)\circ \theta(x\otimes a)+
  \Big((z\otimes c)\circ(x\otimes a) \Big)\circ \theta(y\otimes b)\nonumber\\
  &&=\Big((x\ast y)\otimes(a\bullet b) \Big)\circ (\alpha(z)\otimes\beta(c))+
  \Big((y\ast z)\otimes(b\bullet c) \Big)\circ (\alpha(x)\otimes\beta(a))\nonumber\\
  &&+
  \Big((z\ast x)\otimes(c\bullet a) \Big)\circ (\alpha(y)\otimes\beta(b))
  =\Big((x\ast y)\ast\alpha(z)\Big)\otimes\Big((a\bullet b)\bullet\beta(c)\Big)\nonumber\\
  &&+
  \Big((y\ast z)\ast\alpha(x)\Big)\otimes\Big((b\bullet c)\bullet\beta(a)\Big)
+\Big((z\ast x)\ast\alpha(y)\Big)\otimes\Big((c\bullet a)\bullet\beta(b)\Big)\nonumber\\
&&=\Big((x\ast y)\ast\alpha(z)+(y\ast z)\ast\alpha(x)+(z\ast x)\ast\alpha(y)\Big)\otimes\Big((a\bullet b)\bullet\beta(c)\Big)\nonumber\\
&& \mbox{  ( by the commutativity and the Hom-associativity of $(A,\bullet, \beta)$ )}\nonumber\\
&&=0 \mbox{ ( by (\ref{JJi}) ).}\nonumber
 \end{eqnarray}
 Hence, $(L\otimes A, \circ ,\theta:=\alpha\otimes\beta)$ is a Hom-Jacobi-Jordan algebra.
 \end{proof}
 Similarly, we can also prove the next result.
 \begin{proposition}
  Let $(L, \ast, \alpha)$ be a Hom-Jacobi-Jordan algebra and $(A,\bullet)$ be a commutative associative algebra. Then  $(L\otimes A, \circ ,\theta_0:=\alpha\otimes Id_A)$ is a Hom-Jacobi-Jordan algebra  called  a current Hom-Jacobi-Jordan algebra. Precisely:
 \begin{eqnarray}
  && (x\otimes a)\circ(y\otimes b)=(x\ast y)\otimes(a\bullet b)  \mbox{  $\forall x,y\in L$ and $a,b\in A, $}\nonumber\\
  && \theta_0(x\otimes b)=\alpha(x)\otimes a \mbox{  $\forall x \in L$ and $a \in A.$}\nonumber 
 \end{eqnarray}
 \end{proposition}
\begin{definition}
 Let $(A,\ast,\alpha)$ be Hom-Jacobi-Jordan algebra and $B$ be a vector subspace of $A.$ Then, $(B,\alpha)$ is said to be a Hom-ideal of $(A,\ast,\alpha)$ if for all $(a,b)\in A\times B,$ we have $\alpha(b)\in B$ and $a\ast b\in B.$
\end{definition}
\begin{proposition}
 Let $(A,\ast,\alpha)$ be Hom-Jacobi-Jordan. Then, the set
 $$HAnn(A):=\{b\in A, \alpha(b)=b,\ \forall a\in A, a\ast b=0\}$$
 is a Hom-ideal of $(A,\ast,\alpha)$ with respect to $\alpha$, called the Hom-annihilator of $(A,\ast,\alpha).$
\end{proposition}

Now, we recall the definition of representations of a Hom-Jacobi-Jordan algebra.
\begin{definition} \cite{sa}
 A representation of a Hom-Jacobi-Jordan algebra $(A, \ast, \alpha)$  is a triple $(V,\rho,\phi)$ where $V$ is a vector space, $\phi\in gl(V)$ and $\rho: A\rightarrow gl(V)$ is a linear map
such that the following equalities hold for all $x, y \in A,$
\begin{eqnarray}
\phi\rho(x)=\rho(\alpha(x))\phi, \label{rHJJ1}\\
 \rho(x\ast y)\phi=-\rho(\alpha(x))\rho(y)-\rho(\alpha(y))\rho(x), \mbox{  for all $x,y\in A.$ }\label{rHJJ2}
\end{eqnarray}
\end{definition}
\begin{remark}
 If $\alpha=Id_A,\ \phi=Id_V,$ then $(V,\rho,Id_V)$ reduces to a representation of the Jacobi-Jordan algebra $(A,\cdot, Id_A)$ \cite{alagm,nj}.
\end{remark}
\begin{example} \cite{sa}
 Let $(A,\ast,\alpha)$ be a Hom-Jacobi-Jordan algebra and $(B,\alpha)$ be a  Hom-ideal of $(A,\ast,\alpha).$ Set $\rho(a)b:=a\ast b$ for all $(a,b)\in A\times B,$ then $(B,\alpha)$  is a representation of $(A,\ast,\alpha)$. 
\end{example}
It is  easy to prove the following.
\begin{proposition}
 Let $(A,\ast,\alpha)$ be a Hom-Jacobi-Jordan algebra, $(V_1, \rho_1,\phi_1)$ and $(V_2, \rho_2,\phi_2)$ be two representations of $(A,\ast,\alpha).$ Then 
 $(V_1\oplus V_2, \rho_1\oplus\rho_2,\phi_1\oplus\phi_2)$ is a representation of $(A,\ast,\alpha)$ where
 \begin{eqnarray}
  (\rho_1\oplus\rho_2)(x)(u+v):=\rho_1(x)u+\rho_2(x)v \mbox{ $\forall x\in A$\ and \ $(u,v)\in V_1\times V_2.$}\nonumber
 \end{eqnarray}
\end{proposition}
\begin{proposition}
 Let $(A,\ast,\alpha)$ be a Hom-Jacobi-Jordan algebra, $(V, \rho,\phi)$ be a representation of $(A,\ast,\alpha)$ and 
 $D: A\rightarrow V$ be a linear map. Then 
 $(V_D, \tilde{\rho},\tilde{\phi})$ is a representation of $(A,\ast,\alpha)$ where 
  $$V_D:= V+\mathbb{K}D,\ \tilde{\phi}: V_D\rightarrow V_D \mbox{ and } \tilde{\rho}: A\rightarrow gl(V_D).$$
 with $\forall (r,x,u)\in\mathbb{K}\times A\times V,$
 $$\tilde{\phi}(u+rD):=\phi(u)+r\phi D \mbox{ and } \tilde{\rho}(x)(u+rD)=\rho(x)u+r\rho(x) D.$$
\end{proposition}
\begin{proof}
 Pick $r\in\mathbb{K},\ (x,y)\in A^2$ and $u\in V.$ Then by (\ref{rHJJ1}, we get
 \begin{eqnarray}
  &&\tilde{\phi}\tilde{\rho}(x)(u+rD)=\tilde{\phi}(\rho(x)u+r\rho(x) D)=
  \phi\rho(x)u+r\phi\rho(x)D=\rho(\alpha(x))\phi(u)+r\rho(\alpha(x))\phi D\nonumber\\
  &&=\tilde{\rho}(\alpha(x))(\phi(u)=r\phi D)=\tilde{\rho}(\alpha(x))\tilde{\phi}(u+r D).\nonumber
 \end{eqnarray}
 Hence, $\tilde{\phi}\tilde{\rho}=\tilde{\rho}(\alpha(x))\tilde{\phi}.$ Next, we obtain
 \begin{eqnarray}
  &&\tilde{\rho}(x\ast y)\tilde{\phi}(u+rD)=\tilde{\rho}(x\ast y)(\phi(u)+r\phi D)=\rho(x\ast y)\phi(u)+r\rho(x\ast y)\phi D\nonumber\\
  &&=-\rho(\alpha(x))\rho(y)u-\rho(\alpha(y))\rho(x)u-r\rho(\alpha(x))\rho(y)D-r\rho(\alpha(y))\rho(x)D  \mbox{  ( by  (\ref{rHJJ2})  ) }\nonumber\\
  &&=-\Big(\rho(\alpha(x))\rho(y)u+r\rho(\alpha(x))\rho(y)D\Big)-\Big(\rho(\alpha(y))\rho(x)u+r\rho(\alpha(y))\rho(x)D\Big)\nonumber\\
  &&=-\tilde{\rho}(\alpha(x))(\rho(y)u+r\rho(y)D)-\tilde{\rho}(\alpha(y))(\rho(x)u+r\rho(x)D)\nonumber\\
  &&=-\tilde{\rho}(\alpha(x)\tilde{\rho}(y)(u+rD)-\tilde{\rho}(\alpha(y)\tilde{\rho}(x)(u+rD),\nonumber
 \end{eqnarray}
 i.e., $\tilde{\rho}(x\ast y)\tilde{\phi}=-\tilde{\rho}(\alpha(x)\tilde{\rho}(y)-\tilde{\rho}(\alpha(y)\tilde{\rho}(x).$
\end{proof}
\section{Derivations and anti-derivations of Hom-Jacobi-Jordan algebras}
Let $(A,\ast,\alpha)$ be a  Hom-Jacobi-Jordan algebra. For any nonnegative integer $k$, denote by $\alpha^k$
the $k$-times composition of $\alpha$ i.e., 
$$\alpha^k=\alpha\circ\cdots\circ\alpha \mbox{  ($k$-times)}.$$
In particular, $\alpha^0=Id$ and $\alpha^1=\alpha.$ If $(A,\ast,\alpha)$ is a regular  Hom-Jacobi-Jordan algebra, we denote by $\alpha^{-k}$
the $k$-times composition of the inverse $\alpha^{-1}$.
\begin{definition} \label{DfDer}
Let $(A,\ast,\alpha)$ be a Hom-Jacobi-Jordan algebra and $k$ any nonnegative integer. A linear map $D: A\rightarrow A$ is called an \\
 $\bullet$   $\alpha^k$-derivation of $(A,\ast,\alpha)$ if
  \begin{eqnarray}
   D\circ\alpha=\alpha\circ D, \label{der1}\\
   D(u\ast v)=D(u)\ast\alpha^k(v)+\alpha^k(u)\ast D(v) \mbox{ $\forall u,v\in A.$}\label{der2}
  \end{eqnarray}
$\bullet$   $\alpha^k$-antiderivation of $(A,\ast,\alpha)$ if
  \begin{eqnarray}
   D\circ\alpha=\alpha\circ D, \label{antider1}\\
   D(u\ast v)=-D(u)\ast\alpha^k(v)-\alpha^k(u)\ast D(v) \mbox{ $\forall u,v\in A$.}\label{antider2}
  \end{eqnarray}
\end{definition}
For a regular Hom-Jacobi-Jordan algebra,
$\alpha^{-k}$-(anti)derivations can be defined similarly.\\
\\
Denote by $Der_{\alpha^k}(A)$ (resp. $ADer_{\alpha^k}(A)$), the set of $\alpha^k$-derivation (resp. $\alpha^k$-antiderivation) of a  Hom-Jacobi-Jordan algebra $(A,\ast,\alpha).$ 
\begin{lemma}
 For any $u\in A$ satisfying $\alpha(u)=u$, define 
$D_k(u): A\rightarrow A $ by
$$D_k(u)(v):=u\ast\alpha^k(v), \mbox{    $\forall v\in A$}.$$
Then $D_k(u)$ is an $\alpha^{k+1}$-antiderivation  called an inner $\alpha^{k+1}$-antiderivation of $(A,\ast,\alpha).$
\end{lemma}
\begin{proof}
 Let $v,w\in A.$ Fist, we have by multiplicativity and the condition $\alpha(u)=u,$
 $$\alpha(D_k(u)(v))=\alpha(u)\ast\alpha^{k+1}(v)=u\ast\alpha^k(\alpha(v))=D_k(u)(\alpha(v)).$$
 Next, again by multiplicativity and the condition $\alpha(u)=u,$
 \begin{eqnarray}
  && D_k(u)(v\ast w)=\alpha(u)\ast(\alpha^k(v)\ast\alpha^k(w))\nonumber\\
  &&=-\alpha^{k+1}(v)\ast(\alpha^k(w)\ast u)-\alpha^{k+1}(w)\ast(u\ast\alpha^k(v) \mbox{ ( by (\ref{JJi}) )}\nonumber\\
  &&=-(u\ast\alpha^k(v))\ast\alpha^{k+1}(w)-\alpha^{k+1}(v)\ast(u\ast\alpha^k(w))
  . \nonumber
 \end{eqnarray}
Therefore, $D_k (u)$ is an $\alpha^{k+1}$-anti-derivation. 
\end{proof}
Denote by $IADer_{\alpha^k}(A)$, the set of inner $\alpha^k$-antiderivations of $(A,\ast,\alpha).$ 
$$IADer_{\alpha^k}(A):=\{u\ast\alpha^k(-),\ u\in A\ and \ \alpha(u)=u\}.$$
For any $D_1\in Der_{\alpha^k}(A)\cup ADer_{\alpha^l}(A)$ and $D_2\in Der_{\alpha^r}(A)\cup ADer_{\alpha^s}(A)$ define their  commutator by
$$[D_1,D_2]:=D_1D_2-D_2D_1.$$
Then, it is easy to prove the following.
\begin{lemma}
Let $(A,\ast,\alpha)$ be a Hom-Jacobi-Jordan algebra. Then,
\begin{enumerate}
 \item $[Der_{\alpha^k}(A),Der_{\alpha^s}(A)]\subset Der_{\alpha^{k+s}}(A),$
 \item $[ADer_{\alpha^k}(A),ADer_{\alpha^s}(A)]\subset Der_{\alpha^{k+s}}(A),$
 \item $[ADer_{\alpha^k}(A),Der_{\alpha^s}(A)]\subset ADer_{\alpha^{k+s}}(A).$
\end{enumerate}
\end{lemma}
Denoted by 
$$Der(A):\oplus_{k\geq 0}Der_{\alpha^k}(A).$$
Then as in \cite{jzha}, we an prove that $Der(A)$ is a Lie algebra.\\

For any linear map $D: A\rightarrow A,$ consider the vector space $A+\mathbb{K} D.$ Define a bilinear map $\diamond$ on $A+\mathbb{K} D$ by
$$(u+md)\diamond (v+nD):=u\ast v+mD(v)+nD(u),\ u\diamond v=u\ast v,\ D\diamond u=u\diamond D=D(u)$$
and the linear map $\tilde{\alpha}: A+\mathbb{K} D\rightarrow A+\mathbb{K} D$ by
$$\tilde{\alpha}(u+nD)=\alpha(u)+nD.$$
Then as in \cite{jzha}, we can prove.
\begin{proposition}
 With the above notations, $(A+\mathbb{K} D, \diamond,\tilde{\alpha})$ is a Hom-Jacobi-Jordan algebra if and only if $D$ is an $\alpha$-antiderivation of the 
 Hom-Jacobi-Jordan $(A,\ast,\alpha)$ with 
 $D\circ D=0.$
\end{proposition}
Now, for any $D_1\in Der_{\alpha^k}(A)\cup ADer_{\alpha^l}(A)$ and $D_2\in Der_{\alpha^r}(A)\cup ADer_{\alpha^s}(A)$ define their anti-commutator by
$$\{D_1,D_2\}:=D_1D_2+D_2D_1.$$
In the following two propositions, we give a necessary and sufficient condition for the anti-commutator of two $\alpha$-derivations( respectively $\alpha$-antiderivations) to be an $\alpha$-derivations( respectively $\alpha$-antiderivations).
\begin{proposition} \label{antantid}
 Let $(A,\ast,\alpha)$ be a Hom-Jacobi-Jordan algebra, $D_1\in ADer_{\alpha^k}(A)$ and $D_2\in ADer_{\alpha^s}(A).$ Then 
 $\{D_1,D_2\}\in ADer_{\alpha^{k+s}}(A)$ if and only if
 \begin{eqnarray}
 2\{D_1,D_2\}(u\ast v)= (\alpha^s\circ D_1)(u)\ast(\alpha^k\circ D_2)(v)+(\alpha^k\circ D_2)(u)\ast(\alpha^s\circ D_1)(v)\label{cantiDer}
 \end{eqnarray}
\end{proposition}
\begin{proof}
 Let $D_1\in ADer_{\alpha^k}(A),$  $D_2\in ADer_{\alpha^s}(A)$ and $u,v\in A.$ Then by a straightforward computation, using (\ref{antider1}) and (\ref{antider2}) for $D_1$ and $D_2$ we get
 \begin{eqnarray}
 &&  \{D_1,D_2\}(u\ast v)=D_1D_2(u\ast v)+D_2D_1(u\ast v)= D_1D_2(u)\ast\alpha^{k+s}(v)\nonumber\\
 &&+(\alpha^k\circ D_2)(u)\ast(\alpha^s\circ D_1)(v)+(\alpha^s\circ D_1)(u)\ast(\alpha^k\circ D_2)(v)+\alpha^{k+s}(u)\ast D_1D_2(v)\nonumber\\
 &&+D_2D_1(u)\ast\alpha^{k+s}(v)+(\alpha^s\circ D_1)(u)\ast(\alpha^k\circ D_2)(v)+(\alpha^k\circ D_2)(u)\ast(\alpha^s\circ D_1)(v)\nonumber\\
 &&+\alpha^{k+s}(u)\ast D_2D_1(v)
 =\{D_1,D_2\}(u)\ast\alpha^{k+s}(v)+\alpha^{k+s}(u)\ast\{D_1,D_2\}(v)\nonumber\\
 &&+\Big( (\alpha^s\circ D_1)(u)\ast(\alpha^k\circ D_2)(v)+(\alpha^k\circ D_2)(u)\ast(\alpha^s\circ D_1)(v)\Big).\nonumber
 \end{eqnarray}
 Hence,
 $\{D_1,D_2\}\in ADer_{\alpha^{k+s}}(A)$ if and only if
 $$2\{D_1,D_2\}(u\ast v)= (\alpha^s\circ D_1)(u)\ast(\alpha^k\circ D_2)(v)+(\alpha^k\circ D_2)(u)\ast(\alpha^s\circ D_1)(v).$$
\end{proof}
In the case of two derivations of a Hom-Jacobi-Jordan algebra, we obtain.
\begin{proposition}
 Let $(A,\ast,\alpha)$ be a Hom-Jacobi-Jordan algebra, $D_1\in Der_{\alpha^k}(A)$ and $D_2\in Der_{\alpha^s}(A).$ Then 
 $\{D_1,D_2\}\in Der_{\alpha^{k+s}}(A)$ if and only if
 \begin{eqnarray}
  (\alpha^s\circ D_1)(u)\ast(\alpha^k\circ D_2)(v)+(\alpha^k\circ D_2)(u)\ast(\alpha^s\circ D_1)(v)=0.\label{cDer}
 \end{eqnarray}
\end{proposition}
\begin{proof}
 Let $D_1\in Der_{\alpha^k}(A),$  $D_2\in Der_{\alpha^s}(A)$ and $u,v\in A.$ Then as in Proposition \ref{antantid}, using (\ref{der1}) and (\ref{der2}) for $D_1\in Der_{\alpha^k}(A)$ and $D_2\in Der_{\alpha^s}(A)$, we get
 \begin{eqnarray}
 &&  \{D_1,D_2\}(u\ast v)=D_1D_2(u\ast v)+D_2D_1(u\ast v)= D_1D_2(u)\ast\alpha^{k+s}(v)\nonumber\\
 &&+(\alpha^k\circ D_2)(u)\ast(\alpha^s\circ D_1)(v)+(\alpha^s\circ D_1)(u)\ast(\alpha^k\circ D_2)(v)+\alpha^{k+s}(u)\ast D_1D_2(v)\nonumber\\
 &&+D_2D_1(u)\ast\alpha^{k+s}(v)+(\alpha^s\circ D_1)(u)\ast(\alpha^k\circ D_2)(v)+(\alpha^k\circ D_2)(u)\ast(\alpha^s\circ D_1)(v)\nonumber\\
 &&+\alpha^{k+s}(u)\ast D_2D_1(v)
 =\{D_1,D_2\}(u)\ast\alpha^{k+s}(v)+\alpha^{k+s}(u)\ast\{D_1,D_2\}(v)\nonumber\\
 &&+\Big( (\alpha^s\circ D_1)(u)\ast(\alpha^k\circ D_2)(v)+(\alpha^k\circ D_2)(u)\ast(\alpha^s\circ D_1)(v)\Big).\nonumber
 \end{eqnarray}
 Hence,
 $\{D_1,D_2\}\in Der_{\alpha^{k+s}}(A)$ if and only if
 $$(\alpha^s\circ D_1)(u)\ast(\alpha^k\circ D_2)(v)+(\alpha^k\circ D_2)(u)\ast(\alpha^s\circ D_1)(v).$$
\end{proof}
Similarly to the previous results, on can check the following.
\begin{remark} 
\begin{enumerate}
\item For all $D_1\in Der_{\alpha^k}(A)$  and $D_2\in Der_{\alpha^s}(A),$ 
 $\{D_1,D_2\}\in ADer_{\alpha^{k+s}}$ if and only if (\ref{cantiDer}) holds.
 \item For all $D_1\in ADer_{\alpha^k}(A)$  and $D_2\in ADer_{\alpha^s}(A),$  
 $\{D_1,D_2\}\in Der_{\alpha^{k+s}}$ if and only if (\ref{cDer}) holds.
 \item For all $D_1\in ADer_{\alpha^k}(A)$  and $D_2\in Der_{\alpha^s}(A),$ 
 $\{D_1,D_2\}\in ADer_{\alpha^{k+s}}$ if and only if (\ref{cDer}) holds.
\end{enumerate}
\end{remark}
Next, we introduce  
$\alpha^k$-derivations (respectively $\alpha^k$-antiderivations) of a Hom-Jacobi-Jordan algebra with values in a representation which generalize  $\alpha^k$-derivations (respectively $\alpha^k$-antiderivations) of a Hom-Jacobi-Jordan algebra as follows.
\begin{definition} 
Let $(A,\ast,\alpha)$ be a Hom-Jacobi-Jordan algebra, $(V,\rho,\phi)$ be a representation of $(A,\ast,\alpha)$ and $k$ be any nonnegative integer. A linear map $D: A\rightarrow V$ is called an\\
 $\bullet$   $\alpha^k$-derivation of  $(A,\ast,\alpha)$ with values in $(V,\rho,\phi)$  if
  \begin{eqnarray}
   D\circ\alpha=\phi\circ D, \label{mder1}\\
   D(u\ast v)=\rho(\alpha^k(u))D(v)+\rho(\alpha^k(v))D(u) \mbox{ $\forall u,v\in A.$}\label{mder2}
  \end{eqnarray}
$\bullet$   $\alpha^k$-antiderivation of  $(A,\ast,\alpha)$ with values in $(V,\rho,\phi)$ if
  \begin{eqnarray}
   D\circ\alpha=\phi\circ D, \label{mantider1}\\
   D(u\ast v)=-\rho(\alpha^k(u))D(v)-\rho(\alpha^k(v))D(u) \mbox{ $\forall u,v\in A$} \mbox{ $\forall u,v\in A.$}\label{mantider2}
  \end{eqnarray}
\end{definition}
For a regular Hom-Jacobi-Jordan algebra,
$\alpha^{-k}$-(anti)derivations of  $(A,\ast,\alpha)$ with values in $(V,\rho,\phi)$ can be defined similarly.
\begin{remark}
 Oberve that for any integer $k$,  $\alpha^k$-derivations (respectively $\alpha^k$-antiderivations) of  $(A,\ast,\alpha)$ with values in the adjoint representation $(A,L,\alpha)$ are $\alpha^k$-derivations (respectively $\alpha^k$-antiderivations) of  $(A,\ast,\alpha)$ (see Definition \ref{DfDer} ).
\end{remark}

Similarly, denote by $Der_{\alpha^k}(A,V)$ (resp. $ADer_{\alpha^k}(A,V)$), the set of $\alpha^k$-derivations (resp. $\alpha^k$-antiderivations) of  $(A,\ast,\alpha)$ with values in $(V,\rho,\phi).$ Then, we obtain:
\begin{proposition}
 For any $u\in V$ satisfying $\phi(u)=u$, define 
$D_k(u): A\rightarrow V $ by
$$D_k(u)(v):=\rho(\alpha^k(v))u, \mbox{    $\forall v\in A$}.$$
Then $D_k(u)$ is an $\alpha^{k+1}$-antiderivation of  $(A,\ast,\alpha)$ with values in $(V,\rho,\phi)$ called an inner $\alpha^{k+1}$-antiderivation of  $(A,\ast,\alpha)$ with values in $(V,\rho,\phi).$
\end{proposition}
\begin{proof}
 Let $v,w\in A.$ Fist, we have by (\ref{rHJJ1}) and the condition $\phi(u)=u,$ we have
 $$\phi(D_k(u)(v))=\phi(\rho(\alpha^k(v))u)=\rho(\alpha^{k+1}(v))\phi(u)=\rho(\alpha^k(\alpha(v)))u=D_k(u)(\alpha(v)).$$  Hence, $\phi\circ D_k(u)=D_k(u)\circ \alpha.$
 Next, by multiplicativity and the condition $\alpha(u)=u,$
 \begin{eqnarray}
  && D_k(u)(v\ast w)=\rho(\alpha^k(v)\ast\alpha^k(w))(\alpha(u)\nonumber\\
  &&=-\rho(\alpha^{k+1}(v))\rho(\alpha^k(w))u-\rho(\alpha^{k+1}(w))\rho(\alpha^k(v))u \mbox{ ( by (\ref{rHJJ2}) )}\nonumber\\
  &&=-\rho(\alpha^{k+1}(v))D_k(u)(w)-\rho(\alpha^{k+1}(w))D_k(u)(v).
  \nonumber
 \end{eqnarray}
Therefore, $D_k (u)$ is an $\alpha^{k+1}$-antiderivation of  $(A,\ast,\alpha)$ with values in $(V,\rho,\phi)$. 
\end{proof}
Denote by $IADer_{\alpha^k}(A,V)$, the set of inner $\alpha^k$-antiderivation of  $(A,\ast,\alpha)$ with values in $(V,\rho,\phi).$ 
$$IADer_{\alpha^k}(A,V):=\{\rho(\alpha^k(-))u,\ u\in V\ and \ \phi(u)=u\}$$
and set 
$$Der(A,V):\oplus_{k\geq 0}Der_{\alpha^k}(A,V).$$
Then  $Der(A,V)$ is a vector space called the vector space of $\alpha$-antiderivation  of  $(A,\ast,\alpha)$ with values in $(V,\rho,\phi)$.
\section{Zigzag cohomologies of Hom-Jacobi-Jordan algebras}
In this section basing in a cohomology theory of Jacobi-Jordan algebra \cite{abak}, we develop a cohomology theory for Hom-Jacobi-Jordan algebras. First, we define the differential operators and then define the first and second cohomology groups. Finally, interpretations of these groups are given.
\subsection{Definition of a zigzag cohomologies}
For all $n\in \mathbb{N}^*,$ define $C^n(A,V):=\Big\{ f: A^{\otimes n} \rightarrow V  \mbox{ is \ a \ n-linear\ map }\Big\}$ and 
$C^n_{\alpha,\phi}(A,V):=\Big\{ f\in C^n(A,V), \phi\circ f =f\circ\alpha^{\otimes n}\Big\}.$ Denote by 
$S^n_{\alpha,\phi}(A,V)$ (resp. $A^n_{\alpha,\phi}(A,V)$) the set of symmetric (resp. $\alpha$-skew-symmetric) maps of $C^n_{\alpha,\phi}(A,V)$ that is the set of any element $f\in C^n_{\alpha,\phi}(A,V)$ satisfying
$$f(x_{\sigma(1)},\cdots, x_{\sigma(n)})=f(x_1,\cdots, x_n)$$ for all $\sigma$ in  the symmetric group $\mathcal{S}_n$ ( resp. $$f(x_1,\cdots, x_i, \cdots,x_{j-1},\alpha(x_j), x_{j+1}, \cdots, x_n)\\=-f(x_1,\cdots, x_{i-1}, x_j, x_{i+1}, \cdots,x_{j-1},\alpha(x_i), x_{j+1}, \cdots, x_n)$$ 
for $i\neq j$). Let set $$C^0_{\alpha,\phi}(A,V)(A,V)=A^0_{\alpha,\phi}(A,V)=S^0_{\alpha,\phi}(A,V)=\{v\in V, \phi(v)=v\}.$$
 \\
Furtheremore, for any $n\in\mathbb{N},$ define the  operators 
$$d^n_{\rho}: C^n_{\alpha,\phi}(A,V)\rightarrow C^{n+1}_{\alpha,\phi}(A,V)$$ and $$\delta^n_{\rho}: A^n_{\alpha,\phi}(A,V)\rightarrow C^{n+1}_{\alpha,\phi}(A,V)$$ by $$d^0_{\rho}(v)(x)=\delta^0_{\rho}(v)(x):=\rho(x)(v),$$
\begin{eqnarray}
 && d^n_{\rho} f(x_1, \cdots, x_{n+1})=\sum\limits_{i=1}^{n+1}\rho(\alpha^n(x_i))f(x_1,\cdots,\widehat{x_i},\cdots, x_{n+1})\nonumber\\
 &&+\sum\limits_{1\leq i< j\leq n+1} f(x_i\ast x_j, \alpha(x_1), \cdots, \widehat{\alpha(x_i)}, \cdots, \widehat{\alpha(x_j)}, \cdots, \alpha(x_{n+1})) 
\end{eqnarray}
and
\begin{eqnarray}
 && \delta^n_{\rho} f(x_1, \cdots, x_{n+1})=\sum\limits_{i=1}^{n+1}\rho(\alpha^n(x_i))f(x_1,\cdots,\widehat{x_i},\cdots, x_{n+1})\nonumber\\
 &&-\sum\limits_{1\leq i< j\leq n+1} f(x_i\ast x_j, \alpha(x_1), \cdots, \widehat{\alpha(x_i)}, \cdots, \widehat{\alpha(x_j)}, \cdots, \alpha(x_{n+1})). 
\end{eqnarray}
\begin{lemma}
 With the above notations, for any $(f,g)\in C^n_{\alpha,\phi}(A,V)\times A^n_{\alpha,\phi}(A,V),$ we have 
 $$ d^n_{\rho} f\circ \alpha^{\otimes (n+1)}=\phi\circ d^n_{\rho} f \mbox{ \ and \  }\delta^n_{\rho} g\circ \alpha^{\otimes (n+1)}= \phi\circ\delta^n_{\rho} g. $$
\end{lemma}
\begin{proof}
 Let $(f,g)\in C^n_{\alpha,\phi}(A,V)\times A^n_{\alpha,\phi}(A,V)$  and $(x_1, \cdots, x_{n+1})\in A^{n+1},$ then
 \begin{eqnarray}
 && d^n_{\rho} f(\alpha(x_1), \cdots, \alpha(x_{n+1}))=\sum\limits_{i=1}^{n+1}\rho(\alpha^n(\alpha(x_i)))f(\alpha(x_1),\cdots,\widehat{\alpha(x_i)},\cdots, \alpha(x_{n+1}))\nonumber\\
 &&+\sum\limits_{1\leq i< j\leq n+1} f(\alpha(x_i)\ast \alpha(x_j), \alpha^2(x_1), \cdots, \widehat{\alpha^2(x_i)}, \cdots, \widehat{\alpha^2(x_j)}, \cdots, \alpha^2(x_{n+1}))\nonumber\\
 &&=\sum\limits_{i=1}^{n+1}\phi\rho(\alpha^n(x_i))f(x_1,\cdots,\widehat{x_i},\cdots, x_{n+1})
 +\sum\limits_{1\leq i< j\leq n+1} \phi f(x_i\ast x_j, \alpha(x_1), \cdots, \widehat{\alpha(x_i)}, \cdots,\nonumber\\
 &&\widehat{\alpha(x_j)}, \cdots, \alpha(x_{n+1})) \mbox{  ( using $\phi\circ f =f\circ\alpha^{\otimes n}$ and (\ref{rHJJ1}) )} \nonumber\\
 &&=\phi d^n_{\rho} f(x_1, \cdots, x_{n+1}). \nonumber
\end{eqnarray}
Hence we obtain $ d^n_{\rho} f\circ \alpha^{\otimes (n+1)}=\phi\circ d^n_{\rho} f.$  Similarly, we prove that $\delta^n_{\rho} g\circ \alpha^{\otimes (n+1)}= \phi\circ\delta^n_{\rho} g.$
\end{proof}
\begin{theorem}
 For $n\geq 1,$ we have $d^n_{\rho}\circ\delta^{n-1}_{\rho}=0.$
\end{theorem}
\begin{proof}
 Let $f$ be an $(n-1)$-lineair map in $A^{n-1}_{\alpha,\phi}(A,V).$ We have 
 \begin{eqnarray}
  && d^n_{\rho}\circ \delta^{n-1}_{\rho} f(x_1, \cdots, x_{n+1})= \sum\limits_{i=1}^{n+1}\rho(\alpha^n(x_i))\delta^{n-1}_{\rho} f(x_1,\cdots,\widehat{x_i},\cdots, x_{n+1})\nonumber\\
 &&+\sum\limits_{1\leq i< j\leq n+1} \delta^{n-1}_{\rho} f(x_i\ast x_j, \alpha(x_1), \cdots, \widehat{\alpha(x_i)}, \cdots, \widehat{\alpha(x_j)}, \cdots, \alpha(x_{n+1})).\nonumber
 \end{eqnarray}
First, we proceed for the first sum of the right hand side as follows:
\begin{eqnarray}
&& \sum\limits_{i=1}^{n+1}\rho(\alpha^n(x_i))\delta^{n-1}_{\rho} f(x_1,\cdots,\widehat{x_i},\cdots, x_{n+1})\nonumber\\
&&= \sum\limits_{1\leq i< j\leq n+1}\Big( \rho(\alpha^n(x_i))\rho(\alpha^{n-1}(x_j))+\rho(\alpha^n(x_j))\rho(\alpha^{n-1}(x_i))\Big) f(x_1,\cdots,\widehat{x_i},\cdots,\widehat{x_j}, \cdots,  x_{n+1})\nonumber\\
&&-{\sum\limits_{1\leq i,j, k\leq n+1}}_{i\neq j, i\neq k, j<k } \rho(\alpha^n(x_i))f(x_i\ast x_j, \alpha(x_1), \cdots, \widehat{\alpha(x_i)}, \cdots, \widehat{\alpha(x_j)}, \cdots, \alpha(x_{n+1})). \nonumber
\end{eqnarray}
Next, we proceed for the second sum of the right hand side as follows:
\begin{eqnarray}
 && \sum\limits_{1\leq i< j\leq n+1} \delta^{n-1}_{\rho} f(x_i\ast x_j, \alpha(x_1), \cdots, \widehat{\alpha(x_i)}, \cdots, \widehat{\alpha(x_j)}, \cdots, \alpha(x_{n+1}))\nonumber\\
 &&=\sum\limits_{1\leq i< j\leq n+1}\rho(\alpha^{n-1}(x_i\ast x_j))f(\alpha(x_1),\cdots,\widehat{\alpha(x_i)},\cdots,\widehat{\alpha(x_j)}, \cdots,  \alpha(x_{n+1}))\nonumber\\
 &&+{\sum\limits_{1\leq i,j, k\leq n+1}}_{i\neq j, i\neq k, j<k } \rho(\alpha^n(x_i))f(x_i\ast x_j, \alpha(x_1), \cdots, \widehat{\alpha(x_i)}, \cdots, \widehat{\alpha(x_j)}, \cdots, \alpha(x_{n+1}))\nonumber\\
 &&-{\sum\limits_{1\leq i,j, k\leq n+1}}_{i\neq k, j\neq k, i<j } f((x_i\ast x_j)\ast\alpha(x_k), \alpha(x_1), \cdots, \widehat{\alpha(x_i)}, \cdots, \widehat{\alpha(x_j)}, \cdots, \widehat{\alpha(x_k)}, \cdots, \alpha(x_{n+1}))\nonumber\\
 &&-{\sum\limits_{1\leq i,j, k, l\leq n+1}}_{\{i,j\}\cap\{k,l\}=\emptyset, j\neq k, i<j, k<l } f(x_i\ast x_j, \alpha(x_k)\ast\alpha(x_l), \alpha(x_1), \cdots, \widehat{\alpha(x_i)}, \cdots, \widehat{\alpha(x_j)}, \cdots, \widehat{\alpha(x_k)},\nonumber\\
 &&\cdots, \widehat{\alpha(x_l)}, \cdots,\alpha(x_{n+1}))\nonumber\nonumber\\
 &&=\sum\limits_{1\leq i< j\leq n+1}\rho(\alpha^{n-1}(x_i\ast x_j))\phi f(x_1,\cdots,\widehat{x_i},\cdots,\widehat{x_j}, \cdots,  x_{n+1})\nonumber\\
 &&+{\sum\limits_{1\leq i,j, k\leq n+1}}_{i\neq j, i\neq k, j<k } \rho(\alpha^n(x_i))f(x_i\ast x_j, \alpha(x_1), \cdots, \widehat{\alpha(x_i)}, \cdots, \widehat{\alpha(x_j)}, \cdots, \alpha(x_{n+1}))\nonumber\\
 &&\mbox{ (  by the Hom-Jacobi identity, the $\alpha$-skew-symmetry of $f$ and the condition  $f\circ \alpha^{\otimes (n-1)}=\phi\circ f$).} \nonumber
\end{eqnarray}
Hence,
\begin{eqnarray}
&& d^n_{\rho}\circ \delta^{n-1}_{\rho} f(x_1, \cdots, x_{n+1})\nonumber\\
&&= \sum\limits_{1\leq i< j\leq n+1}\Big( \rho(\alpha^n(x_i))\rho(\alpha^{n-1}(x_j))+\rho(\alpha^n(x_j))\rho(\alpha^{n-1}(x_i))+\rho(\alpha^{n-1}(x_i\ast x_j))\phi\Big) f(x_1,\cdots,\widehat{x_i},\nonumber\\ 
&& \cdots,\widehat{x_j}, \cdots,  x_{n+1})=0 \mbox{ ( by (\ref{rHJJ2}) ).}\nonumber
\end{eqnarray}
\end{proof}
Associated to the representation $\rho$ , we obtain the complex $(C^k_{\alpha,\phi}(A,V), d_{\rho},\delta_{\rho}).$ Denote
the set of closed $k$-Hom-cochains by 
\begin{eqnarray}
 Z^k_{\alpha,\phi}(A, V)=\{f\in C^k_{\alpha,\phi}(A,V), d_{\rho}^k f=0\} 
\end{eqnarray}
and the set of exact $k$-Hom-cochains by
\begin{eqnarray}
 B^k_{\alpha,\phi}(A, V)=\{\delta_{\rho}^{k-1} f, f\in A^{k-1}_{\alpha,\phi}(A,V)\}.
\end{eqnarray}
Then we define the
$k^{th}$ cohomology space of $A$ with values/coefficients in $V$ as the quotient
\begin{eqnarray}
 H^k_{\alpha,\phi}(A, V):=Z^k_{\alpha,\phi}(A, V)/B^k_{\alpha,\phi}(A,V).
 \end{eqnarray}
 \subsection{Low degree cohomology spaces}
 In this subsection, we give an interpretation of the low degree cohomology spaces in terms of
algebraic properties of the Hom-Jacobi-Jordan algebra. In degree zero we have:
\begin{eqnarray}
 H^0_{\alpha,\phi}(A, V)&=&\{m\in V, \phi(m)=m\ and\ d_{\rho}^0m=0\}\nonumber\\
 &=&\{m\in V,\phi(m)=m \ and\ \forall x\in A, \rho(x)m=0\}.\nonumber
\end{eqnarray}
For example if $V=A$ is the adjoint represenetation of $(A,\ast,\alpha),$ we obtain
\begin{eqnarray}
 H^0_{\alpha,\alpha}(A, A)
 &=&\{m\in A,\alpha(m)=m \ and\ \forall x\in A, x\ast m=0\},\nonumber
\end{eqnarray}
i.e., $$H^0_{\alpha,\alpha}(A, A)=HAnn(A)$$
is the Hom-annihilator of $A$.

Now, we shall investigate the first group of cohomology. We have
$$H^1_{\alpha,\phi}(A, V):=Z^1_{\alpha,\phi}(A, V)/B^1_{\alpha,\phi}(A,V)=ADer_{\alpha,\phi}(A,V)/IADer_{\alpha,\phi}(A,V)$$
where
\begin{eqnarray}
 Z^1_{\alpha,\phi}(A, V)&=&\{f\in C_{\alpha,\phi}^1(A,V), \forall (x_1,x_2)\in A^2, f(x_1\ast x_2)=-\rho(\alpha(x_1))f(x_2)-\rho(\alpha(x_2))f(x_1)\}\nonumber\\
 &:=&ADer_{\alpha,\phi}(A,V)\nonumber 
\end{eqnarray}
and 
\begin{eqnarray}
 B^1_{\alpha,\phi}(A,V)&=&\{\delta_{\rho}^0 m, m\in C_{\alpha,\phi}^0(A,V)\}\nonumber\\
 &=& \{f\in C_{\alpha,\phi}^1(A,V), \exists m\in V, \phi(m)=m, \forall x\in A, f(x)=\rho(x)m \}\nonumber\\
 &=& \{\rho(-)m, m\in V, \phi(m)=m\}.\nonumber
\end{eqnarray}
A particular ase is $V=A$ be the adjoint representation of $(A,\ast,\alpha).$ Then
\begin{eqnarray}
 Z^1_{\alpha,\alpha}(A,A)&=&\{f\in C_{\alpha,\alpha}^1(A,A), \forall (x_1,x_2)\in A^2, f(x_1\ast x_2)=-\alpha(x_1)\ast f(x_2)-\alpha(x_2)\ast f(x_1)\}\nonumber\\
 &:=&ADer_{\alpha}(A)\nonumber 
\end{eqnarray}
and 
\begin{eqnarray}
 B^1_{\alpha,\alpha}(A,A)
 &=& \{m\ast(-), m\in A, \alpha(m)=m\}:=IADer_{\alpha}(A).\nonumber
\end{eqnarray}
Hence, the quotient space
is then $H^1_{\alpha,\alpha}(A,A) = OADer_{\alpha}(A)$, the space of outer $\alpha$-antiderivations of $A$.
\begin{example}\label{ecoh1}
 Let compute the first cohomologies space $H^1_{\alpha,\phi}(A,A)$ for the Hom-Jacobi-Jordan algebra of Example \ref{ehjj1}. Let $f$ be a linear map of $A$ such that $f(e_1)=a_1e_1+a_2e_2+a_3e_3,\ f(e_2)=b_1e_1+b_2e_2+b_3e_3,f(e_3)=c_1e_1+c_2e_2+c_3e_3.$ Then,
 $$f\in Z^1_{\alpha,\alpha}(A,A)\Leftrightarrow b_1=c_1=a_2=b_2=c_2=a_3=b_3=0;\ c_3=-4b_2;\ a_1=6b_2.$$ Then, $Z^1_{\alpha,\alpha}(A,A)=\{0\}.$ Next, we prove similarly that $B^1_{\alpha,\alpha}(A,A)=\{0\}.$ Hence, $H^1_{\alpha,\alpha}(A,A)=\{0\}.$
\end{example}
\begin{example}
 Consider the $2$-dimensional Hom-Jacobi-Jordan algebra $(A,\ast,\alpha)$ defined by $\alpha(e_1)=e_1+e_2,\ \alpha(e_2)=e_2$
 with the only non-zero product $e_1\ast e_1=e_2.$
 \begin{enumerate}
  \item Let put $V=\mathbb{K},\ \rho=0$ and $\phi=0.$ Then $(V,\rho, 1)$ is a representation of $(A,\ast,\alpha).$ Let 
  $f=(a\,\ b)$ be a linear form of $A$  with respect to the basis $(e_1,e_2).$ Then $f\in B^1_{\alpha,\phi}(A,\mathbb{K})$ if and only if $b=0.$ Therefore, 
  $B^1_{\alpha,\phi}(A,\mathbb{K})$ is spanned by $\{(1\,\ 0)\}.$ Next, we proved similarly that $Z^1_{\alpha,\phi}(A,\mathbb{K})=\{0\}.$ Hence, $H^1_{\alpha,\phi}(A,\mathbb{K})$ is spanned by $\{(1\,\ 0)\}.$
  \item We compute here, $H^1_{\alpha,\alpha}(A,A).$ Let $f$ be a linear map of 
  $A$ such that $f(e_1)=ae_1+be_2,\ f(e_2)=ce_1+de_2.$ We have $f\in Z^1_{\alpha,\alpha}(A,A)$ if and only if $a=c=d=0.$ Thus $Z^1_{\alpha,\alpha}(A,A)
  $ is spanned by $\{\begin{pmatrix}
	0 & 0 \\
	1 & 0 		
	\end{pmatrix}\}.$ Finally, we obtain 
	$B^1_{\alpha,\alpha}(A,A)
  =\{\begin{pmatrix}
	0 & 0 \\
	0 & 0 		
	\end{pmatrix}\}$  and therefore $H^1_{\alpha,\alpha}(A,A)
  $ is spanned by $\{\begin{pmatrix}
	0 & 0 \\
	1 & 0 		
	\end{pmatrix}\}.$
 \end{enumerate}

\end{example}

\subsection{The trivial representation of Hom-Jacobi-Jordan algebras}
In this section, trivial representations of Hom-Jacobi-Jordan algebras are studied. As an application, it is shown that the central extension of a Hom-Jacobi-Jordan algebra $(A,\ast,\alpha)$  is controlled by the second cohomology of $A$ with coefficients in the trivial representation.
Now let $V=\mathbb{R},$ then we have 
$gl(V)=\mathbb{R}$ and any $\phi\in gl(V)$ is exactly a real number which will be noted by $r.$ Let $\rho : A\rightarrow gl(V)=\mathbb{R}$ be the zero map. It is clear that, $\rho$ is a
representation of the Hom-Jacobi-Jordan algebras $(A,\ast,\alpha)$ with respect to
any $r\in \mathbb{R}.$ Let assume in the sequel that $r=1.$ Then, this representation is called the trivial
representation of the Hom-Jacobi-Jordan algebras $(A,\ast,\alpha)$.
Associated to this representation, the set of $k$-cochains on $A$  denoted
by $C^k(A,\mathbb{R}),$ is the set of $k$-linear maps from $A\times\cdots\times A$ to $\mathbb{R}$ and the one of $k$-Hom-cochains is then given by
 $$C^k_{\alpha}(A,\mathbb{R}):=\Big\{ f\in C^k(A,\mathbb{R}), f\circ\alpha^{\otimes k}=f\Big\}$$ 
  with $$C^0_{\alpha,}(A,\mathbb{R})=A^0_{\alpha}(A,\mathbb{R})=\mathbb{R}.$$
 \\
The corresponding operators 
$$d^n: C^n_{\alpha,\phi}(A,\mathbb{R})\rightarrow C^{n+1}_{\alpha,\phi}(A,\mathbb{R})$$ and $$\delta^n: A^n_{\alpha,\phi}(A,\mathbb{R})\rightarrow C^{n+1}_{\alpha,\phi}(A,\mathbb{R})$$ is defined by $$d^0(v)(x)=\delta^0(v)(x):=\rho(x)(v),$$
\begin{eqnarray}
 && d^n f(x_1, \cdots, x_{n+1})=\sum\limits_{1\leq i< j\leq n+1} f(x_i\ast x_j, \alpha(x_1), \cdots, \widehat{\alpha(x_i)}, \cdots, \widehat{\alpha(x_j)}, \cdots, \alpha(x_{n+1})) 
\end{eqnarray}
and
\begin{eqnarray}
 && \delta^n f(x_1, \cdots, x_{n+1})=-\sum\limits_{1\leq i< j\leq n+1} f(x_i\ast x_j, \alpha(x_1), \cdots, \widehat{\alpha(x_i)}, \cdots, \widehat{\alpha(x_j)}, \cdots, \alpha(x_{n+1})).
\end{eqnarray}
Now, consider central extensions of the Hom-Jacobi-Jordan algebra $(A,\ast,\alpha)$. It is clear that, 
$(\mathbb{R},0,1)$ is an abelian Hom-Jacobi-Jordan algebra  with the trivial multiplication and the identity morphism. Let $\theta\in C^2_{\alpha}(A,\mathbb{R}) $    such that $\theta(u,v)=\theta(v,u)$ (symmetric) and consider the direct sum 
$\mathcal{A}=A\oplus\mathbb{R}$ with
the following multiplication and linear map
\begin{eqnarray}
 \mu_{\theta}(u+s,v+t)=u\ast v+\theta(u,v),\\
\tilde{\alpha}(u+s)=\alpha(u)+s.
\end{eqnarray}
\begin{theorem}
 The triple $(\mathcal{A},\mu_{\theta} ,
 \tilde{\alpha})$ is a Hom-Jacobi-Jordan algebras  if and
only if $\theta\in C^2(A,\mathbb{R})$ satisﬁes
$$d^2\theta=0.$$
\end{theorem}
This Hom-Jacobi-Jordan algebras  $(\mathcal{A},\mu_{\theta} ,
 \tilde{\alpha})$ is called the central extension
of $(A,\ast,\alpha)$ by the abelian Hom-Jacobi-Jordan algebras $(\mathbb{R},0,1).$
\begin{proof}
 Similar to the one in \cite{jzha}.
\end{proof}
Also, similarly to \cite{jzha}, we can prove:
\begin{proposition}
 For $\theta_1, \theta_2\in C^2(A,\mathbb{R}),$  if 
 $\theta_1-\theta_2$ is exact, the corresponding two central
extensions $(\mathcal{A},\mu_{\theta_1} ,
 \tilde{\alpha})$ and $(\mathcal{A},\mu_{\theta_2} ,
 \tilde{\alpha})$ are isomorphic.
\end{proposition}
\subsection{The adjoint representations of Hom-Jacobi-Jordan algebras}
Let $(A,\ast,\alpha)$  be a regular Hom-Jacobi-Jordan algebra and let consider the adjoint representation of $A$. An interesting fact that we will see in the sequel, is that there exists many adjoint representations of a Hom-Jacobi-Jordan algebra.
\begin{definition}
 For any integer $s,$ the $\alpha^s$-adjoint representation of the regular Hom-Jacobi-Jordan algebra $(A,\ast,\alpha)$, which we denote by $ad_s,$  is deﬁned by
 \begin{eqnarray}
  ad_s(u)(v):=\alpha^s(u)\ast v \mbox{ $\forall u, v\in A.$ }
 \end{eqnarray}
\end{definition}

\begin{lemma}
 With the above notations, it is obvious to prove:
 \begin{eqnarray}
  ad_s(\alpha(u))\circ\alpha=\alpha\circ ad_s,\\
  ad_s(u\ast v)\circ\alpha=-ad_s(\alpha(u))\circ ad_s(v)-ad_s(\alpha(v))\circ ad_s(u).
 \end{eqnarray}
Thus the deﬁnition of $\alpha^s$-adjoint representation is well deﬁned.
\end{lemma}
 The set of $k$-Hom-cochains is
given by
 $$C^n_{\alpha}(A,A):=\Big\{ f\in C^n(A,A), f\circ\alpha^{\otimes n}=\alpha\circ f\Big\}.$$ 
  with $$C^0_{\alpha,}(A,A)=\{v\in A, \alpha(v)=v\}=A^0_{\alpha}(A,A).$$

Associated to the $\alpha^s$-adjoint representation the corresponding operators:
$$d^n_s: C^n_{\alpha,\phi}(A,A)\rightarrow C^{n+1}_{\alpha,\phi}(A,A)$$ and $$\delta^n_{s}: A^n_{\alpha,\phi}(A,A)\rightarrow C^{n+1}_{\alpha,\phi}(A,A)$$ by $$d^0_s(v)(x)=\delta^0_s(v)(x):=\rho(x)(v),$$
\begin{eqnarray}
 && d^n_s f(x_1, \cdots, x_{n+1})=\sum\limits_{i=1}^{n+1}\alpha^{n+s}(x_i)\ast f(x_1,\cdots,\widehat{x_i},\cdots, x_{n+1})\nonumber\\
 &&+\sum\limits_{1\leq i< j\leq n+1} f(x_i\ast x_j, \alpha(x_1), \cdots, \widehat{\alpha(x_i)}, \cdots, \widehat{\alpha(x_j)}, \cdots, \alpha(x_{n+1})) 
\end{eqnarray}
and
\begin{eqnarray}
 && \delta^n_s f(x_1, \cdots, x_{n+1})=\sum\limits_{i=1}^{n+1}\alpha^{n+s}(x_i)\ast f(x_1,\cdots,\widehat{x_i},\cdots, x_{n+1})\nonumber\\
 &&-\sum\limits_{1\leq i< j\leq n+1} f(x_i\ast x_j, \alpha(x_1), \cdots, \widehat{\alpha(x_i)}, \cdots, \widehat{\alpha(x_j)}, \cdots, \alpha(x_{n+1})). 
\end{eqnarray}
For the $\alpha^s$-adjoint representation , $ad_s$ we obtain the $\alpha^s$-adjoint complex 
$(C_{\alpha}(A,A), d_s,\delta_s ).$
\begin{proposition}
 Associated to the $\alpha^s$-adjoint representations $ad_s$ of the regular regular Hom-Jacobi-Jordan algebra $(A,\ast,\alpha),$ $D\in C_{\alpha}(A,A)$ is a $1$-cocycle if and
only if $D$ is an $\alpha^{s+1}$-anti-derivation. i.e. $D\in ADer_{\alpha^{s+1}}(A).$
\end{proposition}
\begin{proof}
Similar to the one in \cite{jzha}
\end{proof}
\section{Deformations of Hom-Jacobi-Jordan algebras and a relative Rota-Baxter operators }

\subsection{Cohomologies of a relative Rota-Baxter operator on a Hom-Jacobi-Jordan algebra}
\begin{definition}\cite{sa}
 Let $(V,\rho,\phi )$ be a representation of a Hom-Jacobi-Jordan algebra $(A,\ast,\alpha).$ A linear
map $T : V\rightarrow A$ is called a relative Rota-Baxter operator with respect  to  $(V,\rho,\phi )$ if it satisfies 
\begin{eqnarray}
 && T\phi=\alpha T \label{rbHJJ1},\\
  && T(u)\ast T(v)=T\Big(\rho(T(u))v+\rho(T(v))u\Big) \mbox{ for all $u,v\in V.$} \label{rbHJJ2}
\end{eqnarray}
\end{definition}
Observe  that Rota-Baxter operators on Hom-Jacobi-Jordan algebras are relative Rota-Baxter operators with respect to the regular representation.
\begin{proposition}\cite{sa}
 Let $(A,\ast, \alpha)$ be a Hom-Jacobi-Jordan algebra and $T: A\rightarrow V$ a relative Rota-Baxter operator with respect to a representation $(V, \rho, \phi).$ 
 Then $(V, \ast_T, \phi)$ is a Hom-Jacobi-Jordan algebra, where
\begin{eqnarray}
 u\ast_T v:=\rho(T(u))v+\rho(T(v))u \mbox{ for $u,v\in V$.} \label{revasHJJ1}
\end{eqnarray}
\end{proposition}
\begin{proposition}\cite{sa}
 Let $T : V\rightarrow A$ be a relative Rota-Baxter operator on the Hom-Jacobi-Jordan algebra $( A, \ast,\alpha)$ with respect to the representation $(V,\rho, \phi).$
Let define a map $\rho_T : V\rightarrow  gl(A)$ by
 \begin{eqnarray}
  \rho_T(u)x:=T(u)\ast x-T(\rho(x)u) \mbox{ for all $(u,x)\in V\times A.$}\nonumber
 \end{eqnarray}
Then, the triple $(A,\rho_T,\alpha)$ is a representation of the  Hom-Jacobi-Jordan algebra $(V,\ast_T,\phi).$
\end{proposition}
Let $$d_{\rho_T}^n: C^n_{\phi,\alpha}(V,A)\rightarrow C^{n+1}_{\phi,\alpha}(V,A)$$ and $$\delta_{\rho_T}^n: A^n_{\phi,\alpha}(V,A)\rightarrow C^{n+1}_{\phi,\alpha}(V,A)$$ 
be the corresponding  coboundary operators of the Hom-Jacobi-Jordan algebra $(V,\ast_T,\phi)$ with coefficients in the representation $(A,\rho_T,\alpha).$
More precisely,
 $$d_{\rho_T}^0(x)(v)=\delta_{\rho_T}^0(x)(v):=\rho_T(v)(x)=T(v)\ast x-T(\rho(x)v).$$
Furtheremore, for any $n\in\mathbb{N}\setminus\{0\},$ 
\begin{eqnarray}
 && d_{\rho_T}^n f(u_1, \cdots, u_{n+1})=\sum\limits_{i=1}^{n+1}\rho_T(\phi^n(u_i))f(u_1,\cdots,\widehat{u_i},\cdots, u_{n+1})\nonumber\\
 &&+\sum\limits_{1\leq i< j\leq n+1} f(u_i\ast_T u_j, \phi(u_1), \cdots, \widehat{\phi(u_i)}, \cdots, \widehat{\phi(u_j)}, \cdots, \phi(u_{n+1})) 
\end{eqnarray}
and
\begin{eqnarray}
 && \delta_{\rho_T}^n f(u_1, \cdots, u_{n+1})=\sum\limits_{i=1}^{n+1}\rho_T(\phi^n(u_i))f(u_1,\cdots,\widehat{u_i},\cdots, u_{n+1})\nonumber\\
 &&-\sum\limits_{1\leq i< j\leq n+1} f(u_i\ast_T u_j, \phi(u_1), \cdots, \widehat{\phi(u_i)}, \cdots, \widehat{\phi(u_j)}, \cdots, \phi(u_{n+1})). 
\end{eqnarray}
\begin{definition} 
Let $T$ be a relative Rota-Baxter operator of the Hom-Jacobi-Jordan algebra $(A,\ast,\alpha)$ with respect to a representation 
$(V,\rho,\phi).$ The complex cochain $(C^{\bullet}_{\phi,\alpha}(V,A)=\bigoplus_{k=0}^{+\infty} C_{\phi,\alpha}^k(V,A),d_{\rho_T}^n, \delta_{\rho_T}^n)$ is
called the cohomology complex for the relative Rota-Baxter operator $T.$
\end{definition}
The $p^{th}$ cohomology space of $V$ with values/coefficients in $A$ is given 
by the quotient
\begin{eqnarray}
 H^k_{\phi,\alpha}(V,A):=Z^k_{\phi,\alpha}(V,A)/B^k_{\phi,\alpha}(V,A).
 \end{eqnarray}
It is obvious that $x\in A$ is closed if and only if $\alpha(x)=x$ and $L_x\circ T-T\circ p(x)=0$ and $f\in C_{\phi,\alpha}^1(V,A)$ is closed if and only if
$\alpha(T(u))\ast f(v)+\alpha(T(v))\ast f(u)-T(\rho(f(u))(\phi(v))+\rho(f(v))(\phi(u)))-f(\rho(T(u))(v)+\rho(T(v))(u))=0.$
\subsection{Linear deformations of a relative Rota-Baxter operator.}
This subsection  devoted to a 
linear deformations study of relative Rota-Baxter operators using the cohomology theory given for Hom-Jacobi-Jordan algebras. 
\begin{definition}
 Let $(A,\ast,\alpha)$ be a Hom Jacobi-Jordan algebra and $ T:V\rightarrow A $ be a relative Rota-Baxter operator with respect to a representation $(V,\rho,\phi)$. A linear map  $\mathcal{Z}:V\rightarrow A$ is said to generate a linear deformation of $T$ if $ \forall t \in \mathbb{K},T_{t}=T+t\mathcal{Z}$ is a relative Rota-Baxter  operateur of $(A,\ast,\alpha)$ with respect to $(V,\rho,\phi).$
\end{definition}
Observe that a linear map $ \mathcal{Z}:V\rightarrow A $ generates a linear deformation of $T$ if only and if : 
\begin{eqnarray}
&& \mathcal{Z}\phi = \alpha\mathcal{Z}, \label{defor1}\\
&&\mathcal{Z}u \ast \mathcal{Z}v= \mathcal{Z}(\rho(\mathcal{Z}u)v+\rho(\mathcal{Z}v)u),\label{defor2}\\
&& Tu\ast \mathcal{Z}v + Tv\ast \mathcal{Z}u-T(\rho(\mathcal{Z}(u))v+\rho(\mathcal{Z}(v))u)-\mathcal{Z}(\rho(T(u))v+\rho(T(v))u)
=0.\label{defor3}
\end{eqnarray}
The identities $(\ref{defor1})$ and$(\ref{defor1})$ mean that $\mathcal{Z}$ is a relative Rota-Baxter operateur of $(A,\ast,\alpha)$ with respect to $(V,\rho,\phi).$ Observe that (\ref{defor3}) is equivalent to
\begin{eqnarray*}
 \mathcal{Z}(u\ast_Tv)=\rho_T(u)\mathcal{Z}v+\rho_T(v)\mathcal{Z}u.
\end{eqnarray*}
Therefore, $(\ref{defor1})$ and $(\ref{defor3})$ mean that $\mathcal{Z}\in Der_{\phi^0}(V,A)$ i.e. $\mathcal{Z}$ is a $\phi^0$-derivation of the Hom-Jacobi-Jordan algebra $(V,\ast_T,\phi)$ with values in the representation $(A,\rho_T,\alpha).$
\begin{definition}
 Let $(A,\ast,\alpha)$ be a regular Hom-Jacobi-Jordan algebra. A bilinear map $\psi\in C^2_{\alpha}(A,A)$ is said to generate a linear  deformation of the multiplication $\ast$ of the regular Hom-Jacobi-Jordan algebra $(A,\ast,\alpha)$ if the $t$-parametrized family of bilinear operations
\begin{eqnarray}
 \mu_t(u, v)=u\ast v+t\psi(u,v) \mbox{ for $t\in\mathbb{K}$ }
\end{eqnarray}
gives to $(A,\alpha)$ the structure of Hom-Jacobi-Jordan algebras i.e. $(A,\mu_t,\alpha)$ is a Hom-Jacobi-Jordan algebra for all $t\in \mathbb{K}.$
\end{definition}
By computing the Hom-Jordan-Jacobi identity of $\mu_t$ and using Hom-Jacobi identity for $\ast$ we obtain:
\begin{eqnarray}
 &&t\Big(\alpha(x)\ast\psi(y,z)+\alpha(y)\ast\psi(z,x)+\alpha(z)\ast\psi(x,y)+\psi(\alpha(x),y\ast z)+\psi(\alpha(y),z\ast x)\nonumber\\
 &&+\psi(\alpha(z),x\ast y)\Big)+t^2\Big(\psi(\psi(x,y),\alpha(z))+\psi(\psi(y,z),\alpha(x))+\psi(\psi(z,x),\alpha(y))\Big).\nonumber
\end{eqnarray}
Hence, $(A,\mu_t,\alpha)$ is a Hom-Jacobi-Jordan algebra for all $t\in \mathbb{K}$ if and only if:
\begin{eqnarray}
&&\alpha\psi=\psi\alpha^{\otimes 2},\label{defor4a}\\
&&\phi(x,y)=\psi(y,x), \label{defor4b}\\
&& \psi(\psi(x,y),\alpha(z))+ \psi(\psi(y,z),\alpha(x))+ \psi(\psi(z,x),\alpha(y))=0, \label{defor5}\\
&&\alpha(x)\ast\psi(y,z)+\alpha(y)\ast\psi(z,x)+\alpha(z)\ast\psi(x,y)\nonumber\\
&&+\psi(\alpha(x),y\ast z)+\psi(\alpha(y),z\ast x)+\psi(\alpha(z),x\ast y)=0.\label{defor6}
 \end{eqnarray}
 Obviously, (\ref{defor4a}), (\ref{defor4b}) and (\ref{defor5}) means that $\psi$ must itself defines a Hom-Jacobi algebra structure on $(A,\alpha)$.
Furthermore, (\ref{defor4a}) and (\ref{defor6}) mean that $\psi$ is closed with respect to the $\alpha^{-1}$ -adjoint representation
$ad_{-1}$ , i.e. $d^2_{-1}\psi=0.$
 \begin{definition}
  Let $(A, \ast,\alpha)$ be a regular Hom-Jacobi-Jordan algebra. Two linear deformations 
  $\mu^1_t=\ast+t\psi_1$ and $\mu^2_t=\ast+t\psi_2$ are said to be equivalent if there exists
a linear operator $N\in gl(A)$ such that 
$T_t=Id+tN$ is a Hom-Jacobi-Jordan algebra
morphism from $(A, \mu^2_t,\alpha)$ to $(A, \mu^1_t,\alpha)$. In particular, a linear deformation
$\mu_t=\ast+t\psi$ of a regular Hom-Jacobi-Jordan algebra $(A, \ast,\alpha)$ is said to be trivial if there
exists a linear operator $N\in gl(A)$ such that for all $t\in\mathbb{K}$
$T_t=Id+tN$ is  a Hom-Jacobi-Jordan algebra
morphism from $(A, \mu_t,\alpha)$ to $(A, \ast,\alpha).$
 \end{definition}
 For all $t\in\mathbb{K}$
$T_t=Id+tN$ being a
morphism of Hom-algebras is equivalent to
\begin{eqnarray}
&&N\alpha=\alpha N, \\
 &&\psi_2(x,y)-\psi_1(x,y)= x\ast N(y)+N(x)\ast y-N(x\ast y), \label{eqdf1}\\
 &&\psi_1(x,N(y))+\psi_1(N(x),y)=N(\psi_2(x,y))-N(x\ast y),\nonumber\\
 &&\psi_1(N(x),N(y))=0.
\end{eqnarray}
Observe that (\ref{eqdf1}) means that $\psi_2-\psi_1=\delta_{-1} N\in B^2(A,A).$ Hence, it follows:
\begin{theorem}
Let $(A, \ast,\alpha)$ be a regular Hom-Jacobi-Jordan algebra. If two linear deformations 
  $\mu^1_t=\ast+t\psi_1$ and $\mu^2_t=\ast+t\psi_2$ are equivalent, then $\psi_1$ and $\psi_2$ are in the
same cohomology class of $H_{\alpha}^2(A, A)$.
\end{theorem}
\begin{definition}
Let $(A, \ast,\alpha)$ be a regular Hom-Jacobi-Jordan algebra.
 A linear operator $N\in C^1_{\alpha}(A,A)$ is called a Nijienhuis operator if
\begin{eqnarray}
 N(u)\ast N(v)=N(u\ast_N v)
 \end{eqnarray}
 where
 \begin{eqnarray}
u\ast_N v:=N(u)\ast v+u\ast N(v)-N(u\ast v).
\end{eqnarray}
\end{definition}
\begin{theorem}
Let $(A, \ast,\alpha)$ be a regular Hom-Jacobi-Jordan algebra and $N\in C^1_{\alpha}(A,A)$ be a Nijienhuis operator. Then, a deformation $\mu_t$ of $(A,\ast, \alpha)$ can be obtained by putting
\begin{eqnarray}
 \psi(u,v):=\delta_{-1} N(u,v)=u\ast_N v.
 \end{eqnarray}
 Furthermore, this deformation is trivial.
\end{theorem}
\begin{proof}
Similar to the one in \cite{jzha}.
\end{proof}
Now, let give the link between these two deformations.
\begin{proposition}
 Let $(A,\ast, \alpha)$ be a Hom Jacobi-Jordan algebra and $ T:A \rightarrow A$ be a relative Rota-Baxter operator with respect  to a representation $(V,\rho, \phi).$ If a linear map $\mathcal{Z}:V\rightarrow A$ generates a linear deformation of $T$, then the bilinear map 
 $\psi_{\mathcal{Z}} :V^{2}\rightarrow V $ defined by $$\psi_{\mathcal{Z}}(u,v)=\rho(\mathcal{Z}u)v+\rho(\mathcal{Z}v)u, \forall u,v \in V,$$ generates a  linear deformation of the associated Hom-Jacobi-Jordan algebra $(V,\ast_{T},\phi).$ 
\end{proposition}
\begin{proof}
Let denote by $\ast_{T_t}$ the corresponding Hom-Jacobi-Jordan algebra structure associated to the relative
Rota-Baxter operator $T_t:=T + tT.$ Then we obtain
 \begin{eqnarray*}
u\ast_{T_t} v&=&\rho(T_tu)v+\rho(T_tv)u\\
&=&
\rho(Tu)v+t\rho(\mathcal{Z}u)v+\rho(Tv)u+t\rho(\mathcal{Z}v)u\\
&=&u\ast_T v+t\psi_{\mathcal{Z}}(u,v) 
\mbox{\,\ $\forall u,v\in V$.}
 \end{eqnarray*}
Hence, $\psi_{\mathcal{Z}}$ generates a linear deformation of $(V,\ast_{T},\phi).$
\end{proof}
\begin{definition}
 Let $ T $ and $ T' $ be two relative Rota-Baxter operators of $ (A,\ast,\alpha)$ with respect to a representation $ (V,\rho,\phi).$ A morphism from $T'$ to $T$ is a couple $ (\phi_{A},\phi_{V})$ where $\phi_{A}:A\rightarrow A$ is a morphism of Hom-Jacobi-Jordan algebras and $\phi_{V}:V\rightarrow V$ is a linear map such as: 
\begin{eqnarray}
&&\phi_{V}\phi = \phi\phi_{V},\label{morp1}\\
&&T\phi_{V}=\phi_{A}T'\label{morp2},\\
&& \phi_{V}\rho(x)u=\rho(\phi_{A}(x))\phi_{V}(u) \mbox{ $\forall u \in V, \forall x \in A.$}\label{morp3}
\end{eqnarray}
\end{definition}
\begin{remark}
 If $ \phi_{A}$ et $ \phi_{V}$ are invertible, we say that $ (\phi_{A},\phi_{V})$ is an isomorphism from $T$ to $T'.$
\end{remark}
\begin{proposition}
 Let $T$ and $T'$ be relative Rota-Baxter operators on a Hom-Jacobi-Jordan algebra $(A,\ast,\alpha)$ with respect to a representation $(V,\rho,\phi).$
 If $(\phi_{A},\phi_{V})$ is an isomorphism from $T$ to $T'$, then $(\phi_{A}^{-1},\phi_{V}^{-1})$ is an isomorphism from $T'$ to $T.$
\end{proposition}
\begin{proof}
 It is clear that $\phi_{A}^{-1}:A\rightarrow A$ is a morphism of Hom-algebras and $\phi_{V}^{-1}:V\rightarrow V$ is a linear map. Next, we have $ \phi_{V}^{-1}\phi=\phi\phi_{V}^{-1}$
and $\phi_{A}^{-1}T= T'\phi_{V}^{-1}$ by (\ref{morp1}) and (\ref{morp2}) respectively. Finally, let $x\in A$ and $u\in V.$  Then, there exists $(y,v)\in A\times V$ such that $x=\phi_A(y)$ and $u=\phi_V(v)$ i.e. $y=\phi_A^{-1}(x)$ and $v=\phi_V^{-1}(u).$ Hence,
\begin{eqnarray*}
 (\ref{morp3}) &\Leftrightarrow &  \phi_{V}\rho(y)v=\rho(\phi_{A}(y))\phi_{V}(v)\\
 &\Leftrightarrow &
 \phi_{V}\rho(\phi_A^{-1}(x))\phi_V^{-1}(u)=\rho(x)u\\
 &\Leftrightarrow &\phi_{V}^{-1}\rho(x)u=\rho(\phi_A^{-1}(x))\phi_V^{-1}(u).
\end{eqnarray*}
\end{proof}
\begin{proposition}
Let $T$ and $T'$ be relative Rota-Baxter operators on a Hom-Jacobi-Jordan algebra $(A,\ast,\alpha)$ with respect to a representation $(V,\rho,\phi).$
 If $(\phi_{A},\phi_{V})$ is a morphism (resp. isomorphism) from $T$ to $T'$, then $\phi_V$ is a morphism (resp. isomorphism) of Hom-Jacobi-Jordan algebras from $(V,\ast_{T'},\phi)$ to $ (V,\ast_{T},\phi).$
\end{proposition}
\begin{proof}
 Let $u,v\in V.$ It is clear that $\phi\phi_{V}=\phi_{V}\phi.$
 Next, using (\ref{rHJJ1}) and (\ref{rbHJJ1}), we compute:
 \begin{eqnarray*}
&&\phi_{V}(u\ast_{T'}v)=\phi_{V}(\rho(T'u)v+\rho(T'v)u)=\phi_{V}(\rho(T'u)v)+\phi_{V}(\rho(T'v)u)=
\rho(\phi_{A}(T'u))\phi_{V}(v)\\
&&+\rho(\phi_{A}(T'v))\phi_{V}(u)
=\rho(T\phi_{V}(u))\phi_{V}(v)+\rho(T\phi_{V}(v))\phi_{V}(u)
	=\phi_{V}(u)\ast_{T}\phi_{V}(v).
 \end{eqnarray*}
Hence, $ \phi_{V}$ is an homomorphism from $ (V,\ast_{T'},\phi)$ to $ (V,\ast_{T},\phi).$
\end{proof}
\begin{proposition}
 Let $(A,\ast,\alpha)$ be a Hom-Jacobi-Jordan algebra, $T$ be a relative Rota-Baxter operator with respect to a representation $(V,\rho,\phi),$  $ \phi_{A}:A\rightarrow A$ be an isomorphism of Hom-algebras and $\phi_{A}: V\rightarrow V$ an isomorphism of vector spaces such that:
 \begin{eqnarray}
  &&\phi_{V}\phi=\phi\phi_{V}, \label{newT1}\\
  &&\phi_{V}\rho(x)u=\rho(\phi_{A}(x))\phi_{\phi_{V}}(u)\mbox{ $\forall x \in A,\forall u \in V.$}\label{newT2}
 \end{eqnarray}
Then, the map $ T':=\phi_{A}^{-1}T\phi_{V}: V \rightarrow A$ is a relative Rota-Baxter operator with respect to the representation $(V,\rho,\phi).$	
\end{proposition}
\begin{proof}
 First, thanks to conditions (\ref{newT1}), (\ref{rbHJJ1}) and $\alpha\phi_A=\phi_A\alpha$, we get
 \begin{eqnarray*}
  && T'\phi=\phi_{A}^{-1}T\phi_{V}\phi
		=\phi_{A}^{-1}T \phi\phi_{V} =\phi_{A}^{-1}\alpha T\phi_{V} = \alpha\phi_{A}^{-1}T \phi_{V}= \alpha T'.
 \end{eqnarray*}
Next, pick $u,v\in V.$ Then, using $\phi_A$ is a morphism of Hom-algebras, we compute:
\begin{eqnarray*}
&& T'u\ast T'v= (\phi_{A}^{-1}T \phi_{V})(u)\ast (\phi_{A}^{-1}T \phi_{V})(v)= \phi_{A}^{-1}(T\phi_{V}(u))\ast \phi_{A}^{-1}(T\phi_{V}(v)) \\
&&= \phi_{A}^{-1}(T\phi_{V}(u)\ast T\phi_{V}(v))=\phi_{A}^{-1}(T(\rho(T\phi_{V}(u))\phi_{V}(v)+\rho(T\phi_{V}(v))\phi_{V}(u)))\mbox{ by (\ref{rbHJJ2}) })\\
&&= \phi_{A}^{-1}T( \rho(T\phi_{V}(u))\phi_{v}(v)+\rho(T\phi_{V}(v))\phi_{V}(u)) =\phi_{A}^{-1}T( \rho(\phi_{A}T'u)\phi_{V}(v)\\
 &&+\rho(\phi_{A}T'v)\phi_{V}(u)) \mbox{ ( by $T\phi_V=\phi_AT'$  )}\\
&&= \phi_{A}^{-1}T( \phi_{V}\rho(T'u)v+\phi_{V}\rho(T'v)u) (\mbox{ by (\ref{newT2}) })\\
&&=\phi_{A}^{-1}T \phi_{V}( \rho(T'u)v+\rho(T'v)u)=T'(\rho(T'u)v+\rho(T'v)u).
\end{eqnarray*}
\end{proof}
\subsection{Formal deformations of Hom-Jacobi-Jordan algebras}
Given a $\mathbb{K}$-vector space $A$,  
$\mathbb{K}[[t]]$ designates the power series ring in one variable $t$ 
whose coefficients are elements of
$\mathbb{K}$ and $A[[t]]$ be the set of formal series with coefficients in  the vector space $A,$ ($A[[t]]$ is gotten by extending the coefficients domain of $A$
from $\mathbb{K}$ to $\mathbb{K}[[t]]$). Any $\mathbb{K}$-bilinear map $f: A\times A\rightarrow A$ has naturally an extension to $\mathbb{K}[[t]]$-bilinear map $f : A[[t]]\times A[[t]]\rightarrow A[[t]].$ More precisely if $x=\sum\limits^{+\infty}_{i=0}x_{i}t^{i}$ and $y=\sum\limits^{+\infty}_{j=0}y_{i}t^{j}$ then $f(x,y)=\sum\limits_{i\geq 0, j\geq 0}t^{i+j}f(x_i,y_j).$ For linear maps, the same holds. 
\begin{definition}
 Let $(A,\mu,\alpha)$ be a Hom-Jacobi-Jordan algebra. A one-paramater formal deformation of $(A,\mu,\alpha)$ is given by a $ \mathbb{K}[[t]]$-bilinear map $\mu_{t}: A[[t]]\times A[[t]]\rightarrow A[[t]]$ of the form $$\mu_{t}=\sum\limits^{+\infty}_{i=0}\mu_{i}t^{i}$$ where each $\mu_{i}$ is a $\mathbb{K}$-bilinear map
 $\mu_i: A\times A\rightarrow A$ ( extended to  $ \mathbb{K}[[t]]$-bilinear), $\mu_0=\mu$ such that the following conditions hold:
 \begin{eqnarray}
  && \mu_{t}(x,y)=\mu_{t}(y,x) \mbox{ (symmetry)},\nonumber\\
  && \circlearrowleft_{(x,y,z)}\mu_{t}(\mu_{t}(x,y),\alpha(z))=0 \mbox{ (Hom-Jacobi-Jordan identity).} \label{eqdefor}
 \end{eqnarray}
 The deformation is said to be of order $k\in\mathbb{N}$ if $\mu_{t}=\sum\limits^{k}_{i=0}\mu_{i}t^{i}$ and infinitesimal if $k=1.$
\end{definition}
\begin{remark}
The symmetry of $\mu_t$ is equivalent to the symmetry of
all $\mu_i$ for $i\geq 0.$   
\end{remark}
\begin{example}
 Let $(A,\mu,\alpha)$ be a Hom-Jacobi-Jordan algebra.
\begin{enumerate}
\item The bilinear map $\mu_{t}=\sum\limits^{k}_{i=0}\mu_{i}t^{i}$  defined by $\mu_{0}=\mu$ and $\mu_{i}=0$ for all $i\in\mathbb{N}^*,$ is a formal deformation of $(A,\mu,\alpha)$ of order $0$.
	\item 
	The bilinear map $\mu_{t}=\sum\limits^{k}_{i=0}\mu_{i}t^{i}$ defined by: $\mu_{i}=\mu, \forall i \geq 0$ is  formal deformation of $(A,\mu,\alpha)$.
\end{enumerate}	
\end{example}
The identity (\ref{eqdefor}) is called the deformation equation of the 
Hom-Jacobi-algebra. It is equivalent to
$$\circlearrowleft_{(x,y,z)}\sum\limits_{i\geq 0,j\geq0}\ t^{i+j}\mu_{i}(\mu_j(x,y),\alpha(z))=0,$$
i.e.
$$\circlearrowleft_{(x,y,z)}\sum\limits_{i\geq 0,s\geq0}\ t^{s}\mu_{i}(\mu_{s-i}(x,y),\alpha(z))=0,$$
or
$$\sum\limits_{s\geq0}\ t^{s}\circlearrowleft_{(x,y,z)}\sum\limits_{i\geq 0}\ \mu_{i}(\mu_{s-i}(x,y),\alpha(z))=0,$$
which is equivalent to the following infinite system
\begin{eqnarray}
 \circlearrowleft_{(x,y,z)}\sum\limits_{i\geq 0}\ \mu_{i}(\mu_{s-i}(x,y),\alpha(z))=0 \mbox{ for $s=0, 1, 2,\cdots$}\label{eqdefor1}
\end{eqnarray}
In particular for $s=0$, we have 
$$\circlearrowleft_{(x,y,z)}\mu_{0}(\mu_{0}(x,y),\alpha(z))=0,$$
which is the Hom-Jacobi-Jordan identity of 
$(A,\mu,\alpha)$.\\
The equation, for $s=1,$ leads the equation 
\begin{eqnarray}
 &&\mu(\mu_{1}(x,y),\alpha(z))+\mu(\mu_{1}(y,z),\alpha(x))+\mu(\mu_{1}(z,x),\alpha(y))+\mu_{1}(\mu(x,y),\alpha(z))\nonumber\\
 &&+\mu_{1}(\mu(y,z),\alpha(x))+\mu_{1}(\mu(z,x),\alpha(y))=0,
\end{eqnarray}
 which means that $\mu_1$ is closed with respect to the $\alpha^{-1}$ -adjoint representation
$ad_{-1}$ , i.e. $d^2_{-1}\mu_1=0.$ Hence, 
 $\mu_{1}$ is a $2$-Hom-cocycle i.e 
		$\mu_{1}\in \mathcal{Z}^{2}_{\alpha}(A,A)$.\\
		It is easy to prove the following:
\begin{proposition}
 If $\mu_{t}=\sum\limits^{k}_{i=0}\mu_{i}t^{i}$ is a formal deformation of $(A,\mu,\alpha)$, then the triple $(A,\mu_{t},\alpha)$ is a Hom-Jacobi-Jordan algebra.
\end{proposition}
\begin{definition}
 Let $(A,\mu,\alpha)$ be a Hom-Jacobi-Jordan algebra. Two two one-paramater deformations $A_t=(A,\mu_t,\alpha)$ and $A_t'=(A,\mu_t',\alpha)$ of $A$ where $\mu_t=\sum\limits^{+\infty}_{i=0}\mu_{i}t^{i}$ and $\mu'_{t}=\sum\limits^{+\infty}_{i=0}\mu_{i}'t^{i}$ with $\mu_{0}=\mu'_{0}=\mu,$   are said to be equivalent if there exists a formal automorphism  $(\phi_t)_{t\geq 0}: A[[t]]\rightarrow A[[t]]$ that may be written in the form $$\phi_{t}=\sum\limits^{+\infty}_{k=0}\phi_{k}t^{k}$$ where each $\phi_i\in gl(A)$ ( extended to  $ \mathbb{K}[[t]]$-linear), $\phi_{0}=Id_{A}$ such that: 
 \begin{eqnarray}
  && \phi_{t}(\alpha(x))=\alpha(\phi_{t}(x)),\\
  && \phi_{t}(\mu_{t}(x,y))=\mu'_{t}(\phi_{t}(x),\phi_{t}(y)) \mbox{ $\forall x,y\in A[[t]]$}.\label{eqdeform}
 \end{eqnarray}
 A one-paramater formal deformation $A_t=(A,\mu_{t},\alpha)$ of $(A,\mu,\alpha)$ is said to be trivial if $A_t$ is equivalent to $(A,\mu,\alpha)$ (viewed as an algebra over $A[[t]]$).
\end{definition}
\begin{example}
 Suppose that $A=span\{e\}, \mu(e,e)=e, \alpha=0$. The triple $(A,\mu,\alpha)$ is a Hom-Jacobi-Jordan algebra. We set $\mu_{t}=\sum\limits^{p}_{i=0}\mu_it^{i}$, $\mu'_{t}=\sum\limits^{q}_{j=0}\mu_j't^{j}$, $\mu_0=\mu_0'=\mu$ with $\mu_{i}=0 \ \forall i> p,$ $\mu_{j}'=0\  \forall j> q$ and $\phi_{t}=\sum\limits^{+\infty}_{k=0}\phi_kt^{k}$ with $\phi_0=Id_A.$ Then the deformations $\mu_{t}$ and $\mu'_{t}$ are equivalents if only and if $p=q$. 
\end{example}
The identity (\ref{eqdeform}) may be written for all $x, y\in A$ as:
$$\sum\limits_{i\geq 0,j\geq0}\ t^{i+j}\Big(\phi_{i}(\mu_j(x,y))\Big)-\sum\limits_{i\geq 0,j\geq 0,k\geq 0}\ t^{i+j+k}\Big(\mu_j(\phi_{i}(x),\phi_{k}(y))\Big)=0,$$
i.e.
$$\sum\limits_{i\geq 0,s\geq0}\ t^{s}\Big(\phi_{i}(\mu_{s-i}(x,y))\Big)-\sum\limits_{i\geq 0,j\geq 0, s\geq 0}\ t^{s}\Big(\mu_j(\phi_{i}(x),\phi_{s-i-j}(y))\Big)=0.$$
Therefore,
\begin{eqnarray}
 \sum\limits_{i\geq 0}\Big(\phi_{i}(\mu_{s-i}(x,y))-\sum\limits_{j\geq 0}\ \mu_j(\phi_{i}(x),\phi_{s-i-j}(y))\Big)=0 \mbox{ for $s=0,\ 1,\ 2,\ \cdots$}
\end{eqnarray}
In particular, for $s=0$ we have $\mu_0=\mu_0'$, and for $s=1$ we have
\begin{eqnarray*}
 \phi_0(\mu_1(x,y))+\phi_1(\mu_0(x,y))=
 \mu_0'(\phi_0(x),\phi_1(y))+\mu_0'(\phi_1(x),\phi_0(y))+\mu_1'(\phi_0(x),\phi_0(y)).
\end{eqnarray*}
Since $\phi_0=Id_A,\ \mu_0=\mu_0'=\mu$, in the case of a regular Hom-Jacobi-Jordan algebra $(A,\mu,\alpha)$ we have:
\begin{eqnarray}
 \mu_1'(x,y)=\mu_1(x,y)+\phi_1(\mu(x,y))
 -\mu(x,\phi_1(y))-\mu(\phi_1(x),y),\label{eqdeform1}
\end{eqnarray}
i.e.
\begin{eqnarray*}
 \mu_1'-\mu_1=\delta_{-1}\phi_1.
\end{eqnarray*}
Therefore the two $2$-Hom-cocycles corresponding to the two equivalent deformations are cohomologous.
\begin{definition}
 Let $(A,\mu,\alpha)$ be a regular Hom-Jacobi-Jordan algebra and  $\mu_1 \in \mathcal{Z}^{2}_{\alpha}(A,A)$. Then, the 
 $2$-Hom-cocycle $\mu_1$ is said to be integrable if there exists a family $(\mu_{t})_{t\geq 0}$ such that $\mu_{t}=\sum\limits^{+\infty}_{i=0}\mu_{i}t^{i}$ define a formal deformation of $(A,\mu,\alpha)$.
\end{definition}
Thank to identity (\ref{eqdeform1}), the integrability of $\mu_1$ depends only on its cohomology class. Hence, it is easy to prove the following:
\begin{theorem}
 Let $(A,\mu,\alpha)$ a regular Hom-Jacobi-Jordan algebra and $(\mu_{t})_{t\geq 0}$ be a one-paramater formal deformation of $(A,\mu,\alpha)$. Then there exists  a one-paramater formal deformation  $(\mu'_{t})_{t\geq 0}$ of $(A,\mu,\alpha)$ which is equivalent to $(\mu_{t})_{t\geq 0}$ such that $\mu'_{1} \in \mathcal{Z}^{2}_{\alpha}(A,A)$ and $\mu'_{1} \notin \mathcal{B}^{2}_{\alpha}(A,A)$. In addition, if $\mathcal{H}^{2}_{\alpha}(A,A)=\{0\}$, then any one-paramater formal deformation of $(A,\mu,\alpha)$ is equivalent to the trivial one.
\end{theorem}
\begin{remark}
 The Hom-Jacobi-Jordan algebra whose all formal deformations are trivials are said analytically rigids. The nullity of the second cohomology group gives a sufficient criterion for rigidity.
\end{remark}
\subsection{Formal deformations of  relative Rota-Baxter operators on Hom-Jacobi-Jordan algebras}
\begin{definition}
 Let $(A,\mu,\alpha)$ be a Hom-Jacobi-Jordan algebra, $(V,\rho,\phi)$ be a representation of $(A,\mu,\alpha)$ and  $T: V \rightarrow A$ be a relative Rota-Baxter operator of $(A,\mu,\alpha)$ with respect to $(V,\rho,\phi)$. We call a formal deformation of $T$, the given of a family $(T_{t})_{t\in \mathbb{K}}$ of  relative Rota-Baxter operateurs of $(A,\mu,\alpha)$ where $T_{t}=\sum\limits^{+\infty}_{i=0}T_{i}t^{i}$ with $T_{0}=T$ i.e., a formal deformation of $T$ is the given of af a family $(T_{t})_{t\in \mathbb{K}}$ of linear maps from $V$ to $A$ such that
 \begin{eqnarray}
  && T_{t}\phi=\alpha T_{t}, \mbox{$\forall t\in \mathbb{K},$}\label{defrop1}\\
  && \mu(T_{t}u,T_{t}v)=T_{t}(\rho(T_{t}u)v+\rho(T_{t}v)u), \mbox{$\forall u,v \in V.$}\label{defrop2}
 \end{eqnarray}
\end{definition}
\begin{remark}
 \begin{enumerate}
	\item 
The order of the deformation $(T_{t})_{t\in \mathbb{K}}$ is $k$ if $T_{t}=\sum\limits^{k}_{i=0}T_{i}t^{i}$ where $k$ is a nonnegative integer.
\item 
If $k=1$, the deformation is said to be infinitesimal.	
\end{enumerate}
\end{remark}
The identity (\ref{defrop1}) is equivalent to
\begin{eqnarray}
&& \sum\limits_{i\geq 0}\Big(T_i\phi-\alpha T_i)t^i=0,
\end{eqnarray}
i.e.,
\begin{eqnarray}
 T_i\phi=\alpha T_i \mbox{ for $i=0,1,2,\cdots$} \label{defrop3}
\end{eqnarray}
The identity (\ref{defrop2}) is equivalent to
\begin{eqnarray*}
 \sum\limits_{i\geq 0, j\geq 0}t^{i+j}\Big(\mu(T_iu,T_jv)-T_i(\rho(T_ju)v)+\rho(T_jv)u)\Big)=0,
\end{eqnarray*}
i.e.,
\begin{eqnarray*}
 \sum\limits_{i\geq 0}\Big(\mu(T_iu,T_{s-i}v)-T_i(\rho(T_{s-i}u)v)+\rho(T_{s-i}v)u)\Big)=0 \mbox{ for $s=0,1,2,\cdots$}
\end{eqnarray*}
In particular, for $s=0$ we have 
\begin{eqnarray*}
 \mu(T_0u,T_0v)=T_0\Big(\rho(T_0u)v+\rho(T_0v)u\Big)
\end{eqnarray*}
and therefore $T_0$ is a relative Rota-Baxter operator of $(A,\mu,\alpha)$ with respect to the representation $(V,\rho,\phi).$\\
 For $s=1$ we get
 \begin{eqnarray*}
  && \mu(T_0u,T_1v)-T_0\Big(\rho(T_1u)v+\rho(T_1v)u\Big)
 +\mu(T_1u,T_0v)-T_1\Big(\rho(T_0u)v+\rho(T_0v)u\Big),
 \end{eqnarray*}
 i.e., if set $\mu_{T_0}=\ast_{T_0}$, we obtain
\begin{eqnarray}
 T (u\ast_{T_0} v)=\rho_{T_0}(u)T_1v +\rho_{T_0}(v)T_1u\label{defrop4}.
\end{eqnarray}
Hence, by (\ref{defrop3}) and (\ref{defrop4}), it follows that 
$T_1\in Der_{\phi^0}(V,A)$ i.e. $T_1$ is a $\phi^0$-derivation of the Hom-Jacobi-Jordan algebra $(V,\ast_{T_0},\phi)$ with values in the representation $(A,\rho_{T_0},\alpha).$
\begin{proposition}
 If $(T_{t})_{t\in \mathbb{K}}$ is a formal deformation of $T$, then the  bilinear map $\mu_{T_{t}}$ défined by  $\mu_{T_{t}}(u,v)=\sum\limits^{+\infty}_{i=0}(\rho(T_{i}u)v+\rho(T_{i}v)u)t^{i}$ is a formal deformation of $(V,\mu_{T},\phi)$ where $\mu_{T}(u,v)=\rho(Tu)v+\rho(Tv)u, \forall u,v \in V$.	
\end{proposition}

\end{document}